\documentclass[a4paper,11pt,oneside]{article}
\usepackage{amsmath}            
\usepackage{amsthm}             
\usepackage{eucal}              
\usepackage{mathrsfs}           
\DeclareSymbolFont{rsfscript}{OMS}{rsfs}{m}{n}
\DeclareSymbolFontAlphabet{\mathrsfs}{rsfscript}    
\renewcommand{\mathcal}{\mathrsfs}

\allowdisplaybreaks[3]          

\newcounter{ctheorem}[section]                   
\newtheorem{lemma}[ctheorem]{Lemma}              
\newtheorem{thm}[ctheorem]{Theorem}              
\newtheorem{defin}[ctheorem]{Definition}         
\theoremstyle{definition}
\makeatletter \@addtoreset{ctheorem}{chapter} \makeatother  
\makeatletter \@addtoreset{equation}{section} \makeatother  

\def\roundkern#1#2#3#4{\copy255 \kern-#1\wd255 \vrule width #4\wd255
    height #2\ht255 depth #3\ht255 \kern#1\wd255}
\def\straightletter#1{{\mathord{\rm I\mkern-3.6mu #1}}}

\def\R{\straightletter{R}}
\def\Z{\mathord{\rm Z\mkern-6.1mu Z}}
\newcommand{\order}{\CMcal{O}}                  
\newcommand{\nodes}{\CMcal{A}}                  
\newcommand{\To}{\rightarrow}                   
\newcommand{\dotz}[2]{#1,\dotsc,#2}             
\newcommand{\eps}{\varepsilon}                  
\newcommand{\h}[1]{\widehat{#1}\,}              
\newcommand{\cdotc}{\,\cdot\,}                  
\newcommand{\dx}{\mathrm{d}x}                   
\newcommand{\dy}{\mathrm{d}y}                   
\newcommand{\dt}{\mathrm{d}t}                   
\newcommand{\restrict}[1]{\!\!\mid_{#1}}        
\newcommand{\Ndash}{\nobreakdash\textendash}    
\newcommand{\Mdash}{\nobreakdash\textemdash}    
\DeclareMathOperator{\support}{supp}            
\newcommand{\supp}[1]{\support{#1}}

\newcommand{\set}[1]{\{#1\}}                    
\newcommand{\abs}[1]{\lvert#1\rvert}            
\newcommand{\norm}[1]{\lVert#1\rVert}           

\newcommand{\biggabs}[1]{\biggl\lvert#1\biggr\rvert}

\begin{document}
\begin{center}
\Large{\textbf{Error estimates for interpolation of rough data
using the scattered shifts of a radial basis function}}
\end{center}
\medskip
\begin{center}
R. A. Brownlee\footnote{This author was supported by a studentship from the Engineering and Physical Sciences Research Council.} \\
Department of Mathematics, University of Leicester, Leicester LE1
7RH, England
\end{center}
\medskip
%
%

\begin{abstract}
The error between appropriately smooth functions and their radial
basis function interpolants, as the interpolation points fill out a
bounded domain in $\R^d$, is a well studied artifact. In all of
these cases, the analysis takes place in a natural function space
dictated by the choice of radial basis function{\Mdash}the native
space. The native space contains functions possessing a certain
amount of smoothness. This paper establishes error estimates when
the function being interpolated is conspicuously rough.
\end{abstract}
\textbf{Short title:} Error estimates for interpolation of rough
data.

\newpage
%
%

\section{Introduction}

In this paper we are interested in interpolation of a finite
scattered data set $\nodes \subset \R^d$ by translates of a single
basis function. Of the differing set ups to this problem, the one
preferred in this paper is the following variational formulation.
Firstly, we require a space of continuous functions $\CMcal{Z}$
which carries a seminorm. The minimal norm interpolant to $f:\nodes
\To \R$ on $\nodes$ from $\CMcal{Z}$ is the function $Sf \in
\CMcal{Z}$ which agrees with $f$ on $\nodes$ and has smallest
seminorm amongst all other interpolants to $f$ on $\nodes$ from
$\CMcal{Z}$. The particular space we shall be concerned with is
\begin{equation*}
    \CMcal{Z}^m(\R^d) := \biggl\{f \in \mathrsfs{S}^\prime:\ \h{D^\alpha f} \in
    L_{1,{\text{loc}}}(\R^d),\ \int_{\R^d} w(x) \abs{(\h{D^\alpha
    f})(x)}^2\, \dx < \infty,\ \abs{\alpha}=m\biggr\},
\end{equation*}
which carries the seminorm
\begin{equation*}
    \abs{f}_m := \biggl( \sum_{\abs{\alpha}=m} c_\alpha
    \int_{\R^d} w(x) \abs{(\h{D^\alpha f})(x)}^2\, \dx\biggr)^{1/2},\qquad f \in \CMcal{Z}^m(\R^d).
\end{equation*}
The constants $c_\alpha$ are chosen so that $\sum_{\abs{\alpha}=m}
c_\alpha x^{2\alpha} = \abs{x}^{2m}$, for all $x\in\R^d$.  The
notation $\mathrsfs{S}'$ is used to denote the usual Schwartz space
of distributions. The space $\CMcal{Z}^m(\R^d)$ is christened the
\textit{native space}. The \textit{weight function} function $w:\R^d
\To \R$ is initially chosen to satisfy
\begin{list}{}{}
    \item[(W0)] $w \in C(\R^d\setminus 0)$;
    \item[(W1)] $w(x)>0$ if $x \neq0$;
    \item[(W2)] $1/w \in L_{1,{\text{loc}}}(\R^d)$;
    \item[(W3)] there is a positive $\mu \in \R$ such that
    $(w(x))^{-1}=\order(\abs{x}^{-2\mu})$ as $\abs{x} \To \infty$.
\end{list}{}{}
A consequence of (W0){\Ndash}(W3) is that $\CMcal{Z}^m(\R^d)$ is
complete with respect to $\abs{\cdotc}_m$, and if $m+\mu-d/2>0$ then
$\CMcal{Z}^m(\R^d)$ is embedded in the continuous functions (see
\cite{wayne2}). As the title of this work suggests, we expect this
set up to admit minimal norm interpolants of the form
\begin{equation}\label{primary part of interpolant}
(Sf)(x) := \sum_{a \in \nodes} b_a \psi(x-a),\qquad \mbox{for $x \in
\R^d$,}
\end{equation}
for an appropriate \textit{basis function} $\psi$. We are not
disappointed, but for brevity we omit the details which are well
presented in~\cite{wayne2}. The coefficients $b_a$ in~\eqref{primary
part of interpolant} are determined by the interpolation equations
$(Sf)(a)=f(a)$, $a \in \nodes$. In some situations it may be
necessary to append a polynomial $p$ onto~\eqref{primary part of
interpolant} and take up the ensuing extra degrees of freedom by
satisfying the side conditions:
\begin{equation*}
\sum_{a \in \nodes} b_a q(a) = 0,
\end{equation*}
whenever $q$ is a polynomial of the same degree (or less) as $p$.
The archetypal scenario the author has in mind is
$w(x)=\abs{x}^{2\mu}$ for $x \in \R^d$, where $\mu < d/2$. This
leads to minimal norm interpolants of the form~\eqref{primary part
of interpolant} modulo a polynomial of degree $m$. Here, the
\textit{radial basis function} is $\psi:x\mapsto
\abs{x}^{2m+2\mu-d}\ \log{\abs{x}}$ if $2m+2\mu-d$ is an even
integer or $\psi:x\mapsto\abs{x}^{2m+2\mu-d}$ otherwise.

It is of central importance to understand the behaviour of the error
between a function $f:\Omega \To \R$ and its interpolant as the set
$\nodes$ becomes dense in a bounded domain $\Omega$. The measure of
density we employ is the {\it fill-distance} $h:=\sup_{x \in
\overline{\Omega}} \min_{a \in \nodes} \abs{x-a}$. One finds that
there is a positive constant $\gamma(m)$, independent of $h$, such
that for all $f \in \CMcal{Z}^m(\R^d)$,
\begin{equation*}
\norm{f-Sf}_{L_2(\Omega)} = \order{(h^{\gamma(m)})},\qquad \mbox{as
$h \To 0$}.
\end{equation*}
It is natural to ask: what happens if the function being
approximated does not lie in $\CMcal{Z}^m(\R^d)$? It may well be
that $f$ lies in $\CMcal{Z}^k(\R^d)$, where $k<m$ and $k+\mu-d/2>0$.
The condition $k+\mu-d/2>0$ ensures that $f(a)$ exists for each
$a\in\nodes$, so $Sf$ certainly exists. It is tempting to conjecture
that the new error estimate should be
\begin{equation*}
\norm{f-Sf}_{L_2(\Omega)} = \order{(h^{\gamma(k)})},\qquad \mbox{as
$h \To 0$}.
\end{equation*}
We are conjecturing the same approximation order as if we had
instead approximated $f$ with the minimal norm interpolant to $f$ on
$\nodes$ from $\CMcal{Z}^k(\R^d)$. This is precisely what happens in
the case $w=1$, which was considered by Brownlee \& Light
in~\cite{brownlee1}. In this work, with the aid of a recent result
from~\cite{brownlee2} (Lemma~\ref{norm equivalence}), we employ the
technique used by Brownlee \& Light to extend their work to more
general weight functions. Theorem~\ref{main} is the definitive
result we obtain. The interested reader may enjoy consulting the
related papers~\cite{narcowich2, narcowich3, yoon1, yoon2}.

To close this section we introduce some notation that will be
employed throughout the paper. A domain is understood to be a
connected open set. The support of a function $\phi:\R^d \To \R$,
denoted by $\supp(\phi)$, is defined to be the closure of the set
$\set{x \in \R^d:\ \phi(x) \neq 0}$. We make much use of the linear
space $\Pi_m(\R^d)$ which consists of all polynomials of degree at
most $m$ in $d$ variables. We fix $\ell$ as the dimension of this
space. Finally, when we write $\h{f}$ we mean the Fourier transform
of $f$. The context will clarify whether the Fourier transform is
the natural one on $L_1(\R^d)$, $\widehat{f}(x) :=
(2\pi)^{-d/2}\int_{\R^d} f(t) \mathrm{e}^{-\mathrm{i}xt}\, \dt$, or
one of its several extensions to $L_2(\R^d)$ or $\mathrsfs{S}'$.

%
%

\section{Extension theorems}

In this section we gather a number of useful results, chiefly about
the sorts of extensions which can be carried out on our native
spaces. This will first require us to establish the notion of local
native spaces. To do this, we rewrite the seminorm $\abs{f}_m$ in
its \textit{direct form}{\Mdash}that is, without the Fourier
transform of $f$ appearing explicitly. Let us demand that $w$
satisfies the following additional axioms:
\begin{list}{}{}
    \item[(W4)] $w(y)=w(-y)$  for all $ y \in \R^d $;
    \item[(W5)] $ w(0)=0 $ and $\h{w}(x)\leq 0$ for
    almost all $x \in \R^d $;
    \item[(W6)] $\h{w}$ is a measurable function and for
    any neighbourhood $N$ of
    the origin, $\h{w} \in L^1(\R^d\setminus N)$;
    \item[(W7)] $\abs{\h{w}(y)}=\order(\abs{y}^\lambda)$ as $y \rightarrow 0 $ , where
    $\lambda+d+2>0$.
\end{list}{}{}
Armed with axioms (W1) and (W4){\Ndash}W7) it follows from
\cite{levesleylight} that
\begin{equation}\label{direct form}
\abs{f}_m^2=-\frac{1}{2} \sum_{\abs{\alpha}=m} c_\alpha
    \int_{\R^d} \int_{\R^d} \h{w}(x-y) \abs{(D^\alpha f)(x)-(D^\alpha
    f)(y)}^2\, \dx\dy,\qquad f \in \CMcal{Z}^m(\R^d).
\end{equation}
The notation $C^m_0(\R^d)$ is used for the space of compactly
supported $m$\nobreakdash-times continuously differentiable
functions on $\R^d$. Now, let us define the following space for a
domain $\Omega \subset \R^d$,
\begin{equation*}
    X^m(\Omega):=\Bigl\{f\restrict{\Omega}:\ f \in C^m_0(\R^d),\
    \abs{f}_{m,\Omega}<\infty\Bigr\},
\end{equation*}
where
\begin{equation*}
    \abs{f}_{m,\Omega}:=\biggl( -\frac{1}{2} \sum_{\abs{\alpha}=m} c_\alpha
    \int_\Omega \int_\Omega \h{w}(x-y) \abs{(D^\alpha f)(x)-(D^\alpha
    f)(y)}^2\, \dx\dy \biggr)^{1/2},\qquad f \in X^m(\Omega).
\end{equation*}
A norm is placed on $X^m(\Omega)$ via
\begin{equation*}
    \norm{f}_{m,\Omega}:= \Bigl(
    \norm{f}_{W^m_2(\Omega)}^2+\abs{f}_{m,\Omega}^2\Bigr)^{1/2},\qquad
    f \in X^m(\Omega).
\end{equation*}
The notation $\CMcal{X}^m(\Omega)$ denotes the completion of
$X^m(\Omega)$ with respect to $\norm{\cdotc}_{m,\Omega}$, while
$\CMcal{Y}^m(\Omega)$ denotes the completion of $X^m(\Omega)$ with
respect to $\abs{\cdotc}_{m,\Omega}$. It is these spaces that we
call the \textit{local native spaces}.

We are nearly ready to state our first extension theorem, but first
it is necessary to take on board four additional axioms and
introduce an important type of bounded domain:

\begin{defin}
Let $\Omega_1$ and $\Omega_2$ be domains in $\R^d$, and $\Phi$ a
bijection from $\Omega_1$ to $\Omega_2$.  We say that $\Phi$ is
$m$-smooth if, writing
$\Phi(x)=(\dotz{\phi_1(\dotz{x_1}{x_d})}{\phi_d(\dotz{x_1}{x_d})})$
and $\Phi^{-1}(x) = \Psi(x)
=(\dotz{\psi_1(\dotz{x_1}{x_d})}{\psi_d(\dotz{x_1}{x_d})})$, then
the functions $\dotz{\phi_1}{\phi_d}$ belong to
$C^m(\overline{\Omega}_1)$ and $\dotz{\psi_1}{\psi_d}$ belong to
$C^m(\overline{\Omega}_2)$. Let $\Phi$ be a bijection from $\R^d$ to
$\R^d$. We say $\Phi$ is locally $m$-smooth if $\Phi$ is $m$-smooth
on every bounded domain in $\R^d$.
\end{defin}

\begin{list}{}{}
    \item[(W8)] for every locally $(m+1)$-smooth map $\phi$ on $\R^d$, and every bounded subset $\Omega$ of $\R^d$,
    there is a $C_1>0$ such that $\h{w}(\phi(x) - \phi(y)) \leq C_1 \h{w}(x-y)$,
    for all $x,y \in \Omega$;
    \item[(W9)] there exists a constant $C_2>0$ such that if $x=(x',x_d) \in \R^d$ and $y=(x',y_d) \in \R^d$
    with $\abs{x_d} \geq \abs{y_d}$, then $\h{w}(x) \leq C_2 \h{w}(y)$.
    \item[(W10)] $\int_A \h{w}<0$ whenever $A$ has positive
    measure;
    \item[(W11)] $\h{w}(y)=\h{w}(-y)$ for all $y \in \R^d$.
\end{list}{}{}

\begin{defin}
Let $B= \set{(y_1,y_2,\ldots,y_d) \in \R^d:\ \abs{y_j}<1,\ 1\leq j
\leq d}$, and set $B_+= \set{y\in B:\ y=(y', y_d)\ \text{and}\ y_d
>0}$ and $B_0= \set{y \in B:\ y=(y',y_n)\ \text{and}\ y_n =0}$.
A bounded convex domain $\Omega$ in $\R^d$ with boundary $\partial
\Omega$ will be called a V-domain if the following all hold:
\begin{list}{}{}
\item[(A1)] there exist open sets $\dotz{G_1}{G_N} \subset \R^d$
such that $\partial \Omega \subset \bigcup_{j=1}^N G_j;$
\item[(A2)] there exist locally (m+1)-smooth maps $\phi_j : \R^d
\to \R^d$ such that $\phi_j(B) = G_j$, $\phi_j(B_+)=G_j \cap \Omega$
and $\phi_j(B_0)=G_j \cap \partial \Omega$, $j=1,\ldots,N$;
\item[(A3)] let $\Omega_{\delta}$ be the set of
all points in $\Omega$ whose distance from $\partial \Omega$ is less
than $\delta$.  Then for some $\delta>0$,
\begin{equation*}
\Omega_\delta \subset \bigcup_{j=1}^N \phi_j \biggl(
\biggl\{(y_1,y_2,\ldots,y_d) \in \R^d : \abs{y_j}<\frac{1}{m+1},\;
1\leq j \leq d \biggr\}\biggr).
\end{equation*}
\end{list}{}{}
\end{defin}

The definition of a V-domain is taken from a paper by Light \&
Vail~\cite{vail} in which extension theorems for our local native
spaces are considered.

\begin{thm}[Light \& Vail~\cite{vail}]\label{vail norm exten thm}
Let $\Omega \subset \R^d$ be a V-domain. Let $\h{w}:\R^d \To \R$
satisfy (W6){\Ndash}(W11). Then there exists a continuous linear
operator $L:\CMcal{X}^m(\Omega)\To\CMcal{X}^m(\R^d)$ such that for
all $f \in \CMcal{X}^m(\Omega)$,
\begin{enumerate}
    \item $Lf=f$ on $\Omega$;
    \item $\supp(Lf)$ is compact and independent of $f$;
    \item $\norm{Lf}_{m,\R^d} \leq K \norm{f}_{m,\Omega}$, for
    some positive constant $K=K(\Omega)$ independent of $f$.
\end{enumerate}
\end{thm}

A feature of the construction of the extension operator in
Theorem~\ref{vail norm exten thm} is that $Lf$ can be chosen to be
supported on any compact subset of $\R^d$ containing $\Omega$. For
details of the construction, the reader should consult~\cite{vail}.
Also at our disposal is a seminorm version of Theorem~\ref{vail norm
exten thm}:

\begin{thm}[Light \& Vail~\cite{vail}]\label{vail norm exten thm 2}
Let $\Omega \subset \R^d$ be a V-domain. Let $\h{w}:\R^d \To \R$
satisfy (W6){\Ndash}(W11). Given $f \in \CMcal{Y}^m(\Omega)$, there
exists a function $f^\Omega \in \CMcal{Y}^m(\R^d)$ such that:
\begin{enumerate}
    \item $f^\Omega=f$ on $\Omega$;
    \item $\abs{f^\Omega}_{m,\R^d} \leq C \abs{f}_{m,\Omega}$, for
    some positive constant $C=C(\Omega)$ independent of $f$.
\end{enumerate}
\end{thm}

It is convenient for us to be able to work with a norm on
$\CMcal{X}^m(\Omega)$ that is equivalent to
$\norm{\cdotc}_{m,\Omega}$.

\begin{lemma}[Brownlee \& Levesley~\cite{brownlee2}]\label{norm equivalence}
    Let $\Omega \subset \R^d$ be a V-domain. Let $w:\R^d \To \R$
    satisfy (W0){\Ndash}(W12) and let $m+\mu-d/2>0$. Let
    $\dotz{b_1}{b_\ell} \in \Omega$ be unisolvent with respect to
    $\Pi_m(\R^d)$. Define a norm on $\CMcal{X}^m(\Omega)$ via
        \begin{equation*}
            \norm{f}_\Omega := \biggl(
    \abs{f}_{m,\Omega}^2+\sum_{i=1}^\ell
            \abs{f(b_i)}^2 \biggr)^{1/2},\qquad f\in \CMcal{X}^m(\Omega).
        \end{equation*}
    There are positive constants $K_1$ and
    $K_2$ such that for all $f \in \CMcal{X}^m(\Omega)$,
    $K_1 \norm{f}_{m,\Omega} \leq \norm{f}_{\Omega}  \leq K_2
    \norm{f}_{m,\Omega}$.
\end{lemma}

The behaviour of the constant $K(\Omega)$ in the statement of
Theorem~\ref{vail norm exten thm} can be understood for simple
choices of $\Omega$. To realise this, we require that the weight
function satisfies one further and final axiom:
\begin{list}{}{}
    \item[(W12)] there exists $C_1,C_2>0$ such that $C_1 h^\lambda \h{w}(x) \leq \h{w}(h x)\leq C_2 h^{\lambda} \h{w}(x)$, for all $h>0$, $x \in \R^d$.
\end{list}{}{}
Now, an elementary change of variables gives us:

\begin{lemma}\label{CoV lemma 2}
    Let $\Omega$ be a measurable subset of $\R^d$. Let $w:\R^d \To \R$ be a
    measurable function that is nonpositive almost everywhere and satisfies (W11).
    Define the mapping
    $\sigma:\R^d\rightarrow\R^d$ by
    $\sigma(x)=a+h(x-t)$, where $h>0$, and
    $a$, $t$, $x \in \R^d$. Then there exists a constant $K_1,K_2>0$, independent of $\Omega$,
    such that for all $f \in \CMcal{Y}^m(\sigma(\Omega))$,
        \begin{equation*}
            K_1 \leq \frac{\abs{f \circ \sigma}_{m,\Omega}}{h^{m-\lambda/2-d} \abs{f}_{m,\sigma(\Omega)}} \leq K_2.
        \end{equation*}
\end{lemma}

We are now ready to state the key result of this section, but before
doing this let us make a simple observation. Look at the unisolvent
points $\dotz{b_1}{b_\ell}$ in the statement of Lemma~\ref{norm
equivalence}. Since $\CMcal{X}^m(\Omega)$ can be embedded in
$C(\Omega)$, it makes sense to talk about the interpolation operator
$P:\CMcal{X}^m(\Omega) \rightarrow \Pi_m(\R^d)$ based on these
points.

\begin{lemma}\label{poly extention}
    Let $w:\R^d \To \R$ satisfy (W0){\Ndash}(W12). Let $B$ be any ball
    of radius $h$ and centre $a \in \R^d$, and let
    $f \in \CMcal{X}^m(B)$. Whenever $\dotz{b_1}{b_\ell} \in \R^d$ are unisolvent
    with respect to $\Pi_m(\R^d)$ let $P_b:C(\R^d) \To \Pi_m(\R^d)$ be the
    Lagrange interpolation operator on $\dotz{b_1}{b_\ell}$. Then
    there exists $c=(\dotz{c_1}{c_\ell}) \in B^\ell$ and
    $g \in \CMcal{X}^m(\R^d)$ such that
        \begin{enumerate}
            \item $g(x)=(f-P_c f)(x)$ for all $x \in B$;
            \item $g(x)=0$ for all $\abs{x-a}>2h$;
            \item there exists a $C>0$, independent of $f$ and $B$, such
            that $\abs{g}_{m,\R^d} \leq C \abs{f}_{m,B}$.
        \end{enumerate}
    Furthermore, $\dotz{c_1}{c_\ell}$ can be arranged so that $c_1=a$.
\end{lemma}

\begin{proof}
Let $B_1$ be the unit ball in $\R^d$ and let $B_2=2B_1$. Let
$\dotz{b_1}{b_\ell} \in B_1$ be unisolvent with respect to
$\Pi_m(\R^d)$. Define $\sigma(x)=h^{-1}(x-a)$ for all $x \in \R^d$.
Set $c_i = \sigma^{-1}(b_i)$ for $i=\dotz{1}{\ell}$ so that
$\dotz{c_1}{c_\ell} \in B$ are unisolvent with respect to
$\Pi_m(\R^d)$. Take $f \in \CMcal{X}^m(B)$. Then $(f-P_c f) \circ
\sigma^{-1} \in \CMcal{X}^m(B_1)$. Set $F=(f-P_c f) \circ
\sigma^{-1}$. Let $F^{B_1}$ be constructed as an extension to $F$ on
$B_1$. By Theorem~\ref{vail norm exten thm} and the remark following
it, we can assume $F^{B_1}$ is supported on $B_2$. Define $g=F^{B_1}
\circ \sigma \in \CMcal{X}^m(\R^d)$. Let $x \in B$. Since
$\sigma(B)=B_1$ there is a $y \in B_1$ such that $x=\sigma^{-1}(y)$.
Then,
\begin{equation*}
    g(x)=(F^{B_1} \circ \sigma)(x) = F^{B_1}(y) = ((f-P_c f) \circ
    \sigma^{-1}) (y) = (f-P_c f)(x).
\end{equation*}
Also, for $x \in \R^d$ with $\abs{x-a}>2h$, we have
$\abs{\sigma(x)}> 2$. Since $F^{B_1}$ is supported on $B_2$,
$g(x)=0$ for $\abs{x-a}>2h$. Hence, $g$ satisfies properties
\textit{1} and \textit{2}. By Theorem~\ref{vail norm exten thm}
there is a $K_1$, independent of $f$ and $B$, such that
\begin{equation*}
    \norm{F^{B_1}}_{m,B_2} \leq \norm{F^{B_1}}_{m,\R^d} \leq K_1
    \norm{F}_{m,B_1}.
\end{equation*}
We have seen in Lemma~\ref{norm equivalence} that if we endow
$\CMcal{X}^m(B_1)$ and $\CMcal{X}^m(B_2)$ with the norms
\begin{equation*}
    \norm{v}_{B_i} = \biggl( \abs{v}_{m,B_i}^2 + \sum_{i=1}^\ell
    \abs{v(b_i)}^2 \biggr)^{1/2},
\qquad i=1,2,
\end{equation*}
then $\norm{\cdotc}_{B_i}$ and $\norm{\cdotc}_{m,B_i}$ are
equivalent for $i=1,2$.  Thus, there are constants $K_2$ and $K_3$,
independent of $f$ and $B$, such that
\begin{equation*}
    \norm{F^{B_1}}_{B_2} \leq K_2 \norm{F^{B_1}}_{m,B_2} \leq K_1 K_2
    \norm{F}_{m,B_1} \leq K_1 K_2 K_3 \norm{F}_{B_1}.
\end{equation*}
Set $C_1=K_1 K_2 K_3$. Since $F^{B_1}(b_i)=F(b_i)=(f-P_c
f)(\sigma^{-1}(b_i))=(f-P_c f)(c_i)=0$ for $i=\dotz{1}{\ell}$, it
follows that $\abs{F^{B_1}}_{m,B_2} \leq C_1 \abs{F}_{m,B_1}$. Thus,
$\abs{g \circ \sigma^{-1}}_{m,\R^d} \leq C_1 \abs{(f-P_c f) \circ
\sigma^{-1}}_{m,B_1}$. Now, Lemma~\ref{CoV lemma 2} can be employed
twice to provide us with constants $C_2$ and $C_3>0$, independent of
$f$ and $B$, such that
\begin{multline*}
    \abs{g}_{m,\R^d} \leq C_2 h^{d+\lambda/2-m}\abs{g \circ
    \sigma^{-1}}_{m,\R^d}\\
    \leq C_1 C_2 h^{d+\lambda/2-m} \abs{(f-P_c f) \circ \sigma^{-1}}_{m,B_1}
    \leq C_1 C_2 C_3  \abs{f-P_c f}_{m,B}.
\end{multline*}
Finally, we observe that $\abs{f-P_c f}_{m,B} = \abs{f}_{m,B}$ to
complete the first part of the proof. The remaining part follows by
selecting $b_1=0$ and choosing $\dotz{b_2}{b_\ell}$ accordingly in
the above construction.
\end{proof}

%
%

\section{Error estimates}

In this section we establish the error estimate conjectured in the
introduction. We begin with a function $f$ in $\CMcal{Z}^k(\R^d)$.
We want to estimate
\begin{equation}\label{error}
\norm{f-S_mf}_{L_2(\Omega)},
\end{equation}
where $S_m$ is the minimal norm interpolation operator from
$\CMcal{Z}^m(\R^d)$ on $\nodes$ and $m> k$. The essence of the proof
is as follows. Firstly, by adjusting $f$, we obtain a function
$\tilde{f}$, still in $\CMcal{Z}^k(\R^d)$, with seminorm in
$\CMcal{Z}^k(\R^d)$ not too far from that of $f$. We then smooth
$\tilde{f}$ by convolving it with a function $\phi \in
C^\infty_0(\R^d)$. The key feature of the adjustment of $f$ to
$F:=\phi * \tilde{f}$ is that $F(a) = f(a)$ for every point $a \in
\nodes$ (Theorem~\ref{general result}). This enables us to replace
$S_m f$ with $S_m F$ in~\eqref{error}. Furthermore, it follows that
$F\in \CMcal{Z}^m(\R^d)$ so we can employ an existing $L_2$-error
estimate to $F-S_m F$. The remaining part of the error, $f-F$, is
easily dealt with as it vanishes on $\nodes$. Finally,
Lemma~\ref{seminorm convolution bound} takes us back to an error
estimate in $\CMcal{Z}^k(\R^d)$.

\begin{lemma}\label{seminorm convolution bound}
    Let $w:\R^d \To \R$ satisfy (W0) and (W1). Let $k\leq m$ and
    $\phi \in C^\infty_0(\R^d)$. For each $h>0$ let
    $\phi_h(x) = h^{-d} \phi(x/h)$ for $x \in \R^d$. Then there exists
    a constant $C>0$, independent of $h$, such that for all $f \in
    \CMcal{Z}^k(\R^d)$, $\abs{\phi_h * f}_{m,\R^d} \leq C h^{k-m} \abs{f}_{k,\R^d}$.
\end{lemma}

\proof The case $w=1$ is established in \cite{brownlee1}. The proof
for this more general set up does not differ substantially so is
omitted. \qed

\begin{lemma}[Brownlee \& Light \cite{brownlee1}]\label{poly repro}
    Suppose $\phi \in C^{\infty}_0(\R^{d})$ is supported on the unit ball and
    satisfies
    \begin{equation*}
        \int_{\R^d} \phi(x)\,\dx = 1\qquad \mbox{and}\qquad \int_{\R^d}
        \phi(x) x^\alpha\,\dx =0,\qquad \mbox{for all}\  0< \abs{\alpha}
        \leq k.
    \end{equation*}
    For each $\eps>0$ and $x \in \R^d$, let $\phi_\eps(x) =
\eps^{-d}
    \phi(x/\eps)$. Let $B$ be any ball of radius $h$ and centre $a \in
    \R^d$. For a fixed $p \in \Pi_k(\R^d)$ let $f$ be a mapping from
    $\R^d$ to $\R$ such that $f(x) = p(x)$ for all $x \in B$. Then
    $(\phi_\eps * f) (a) = p(a)$ for all $\eps \leq h$.
\end{lemma}

\begin{defin}
    Let $\Omega$ be an open, bounded subset of $\R^d$. Let $\nodes$ be a
    set of points in $\Omega$. The quantity $h:=\sup_{x \in
    \overline{\Omega}} \inf_{a \in \nodes} \abs{x-a}$ is called the fill-distance of
    $\nodes$ in $\Omega$. The separation of $\nodes$ is
    given by the quantity $q:=\min_{\substack{a,b \in \nodes \\ a\neq b} }
        \frac{\abs{a-b}}{2}$. The quantity $h/q$ will be called the mesh-ratio of $\nodes$.
\end{defin}

\begin{thm}\label{general result}
    Let $w:\R^d \To \R$ satisfy (W0){\Ndash}(W12).
    Let $k+\mu-d/2>0$ and $m\geq k$. Let $\nodes$ be a finite subset
    of $\R^d$ of separation $q>0$. Then
    for all $f \in \CMcal{X}^k(\R^d)$ there exists an $F \in \CMcal{X}^m(\R^d)$
    such that
        \begin{enumerate}
            \item $F(a)=f(a)$ for all $a \in \nodes$;
            \item there exists a $C>0$, independent of $f$ and $q$, with $\abs{F}_{k,\R^d}\leq C \abs{f}_{k,\R^d}$
             and $\abs{F}_{m,\R^d} \leq C q^{k-m}\abs{f}_{k,\R^d}$.
        \end{enumerate}
\end{thm}

\proof Take $f \in \CMcal{X}^k(\R^d)$. For each $a \in \nodes$ let
$B_a \subset \R^d$ denote the ball of radius $\delta = q/4$ centred
at $a$. For each $B_a$ let $g_a$ be constructed in accordance with
Lemma~\ref{poly extention}.  That is, for each $a \in \nodes$ take
$c'=(\dotz{c_2}{c_\ell}) \in B_a^{\ell-1}$ and $g_a \in
\CMcal{X}^k(\R^d)$ such that
\begin{enumerate}
    \item $a,\dotz{c_2}{c_\ell}$ are unisolvent with respect to $\Pi_k(\R^d)$
    \item $g_a(x)=(f-P_{(a,c')} f)(x)$ for all $x \in B_a$;
    \item $P_{(a,c')} f \in \Pi_k(\R^d)$ and $(P_{(a,c')} f)(a) = f(a)$;
    \item $g_a(x)=0$ for all $\abs{x-a}>2\delta$;
    \item there exists a $C_1>0$, independent of $f$ and $B_a$, such
    that $\abs{g_a}_{k,\R^d} \leq C_1 \abs{f}_{k,B_a}$.
\end{enumerate}
Note that if $a\neq b$, then $\supp(g_a)$ does not intersect
$\supp(g_b)$, because if $x\in \supp(g_a)$ then
\begin{equation*}
  \abs{x - b} > \abs{b-a}-\abs{x-a} \geq 2q - 2 \delta = 6 \delta.
\end{equation*}
Let $U = \bigcup_{b \in \nodes} \supp{(g_b)}$, then writing
$\R^d=(\R^d \setminus U) \cup U$ we obtain
\begin{align}
\biggabs{\sum_{a \in \nodes} g_a}_{k,\R^d}^2
    &= -\frac{1}{2}
    \sum_{\abs{\alpha}=k} c_\alpha \int_{\R^d} \int_{\R^d} \h{w}(x-y)
    \biggabs{\sum_{a \in \nodes} ((D^\alpha g_a)(x)-(D^\alpha
    g_a)(y))}^2\, \dx \dy \nonumber \\
        &=   -\frac{1}{2} \sum_{\abs{\alpha}=k} c_\alpha
    \biggl( \int_{\R^d \setminus U} \int_{\R^d \setminus U} \h{w}(x-y) \biggabs{\sum_{a \in \nodes} ((D^\alpha g_a)(x)-(D^\alpha
    g_a)(y))}^2\, \dx \dy \nonumber \\
        & \qquad +2\int_{\R^d \setminus U} \int_{U} \h{w}(x-y) \biggabs{\sum_{a \in \nodes} ((D^\alpha g_a)(x)-(D^\alpha
    g_a)(y))}^2\, \dx \dy \nonumber \\
        & \qquad\qquad\qquad +\int_{U} \int_{U} \h{w}(x-y) \biggabs{\sum_{a \in \nodes}
((D^\alpha g_a)(x)-(D^\alpha
    g_a)(y))}^2\, \dx \dy\biggr)  \label{split double integral}.
\end{align}
We shall now consider each of the double integrals in~\eqref{split
double integral} separately. Firstly, the integral over $(\R^d
\setminus U) \times (\R^d \setminus U)$ is zero because $\sum_{a \in
\nodes} g_a$ is supported on $U$. Next, using the observation above
regarding the support of $g_a$, $a \in \nodes$, it follows that
\begin{align}
-\frac{1}{2} \sum_{\abs{\alpha}=k} c_\alpha &\int_{\R^d \setminus U}
\int_{ U} \h{w}(x-y) \biggabs{\sum_{a \in \nodes} ((D^\alpha
g_a)(x)-(D^\alpha
    g_a)(y))}^2\, \dx \dy \nonumber \\
    &= \sum_{b \in \nodes} -\frac{1}{2} \sum_{\abs{\alpha}=k} c_\alpha
    \int_{\R^d \setminus U} \int_{\supp{(g_b)}} \h{w}(x-y) \biggabs{\sum_{a \in
    \nodes} (D^\alpha g_a)(x)}^2\, \dx \dy \nonumber \\
    &= \sum_{b \in \nodes} -\frac{1}{2} \sum_{\abs{\alpha}=k} c_\alpha
    \int_{\R^d \setminus U} \int_{\supp{(g_b)}} \h{w}(x-y) \abs{(D^\alpha g_b)(x)}^2\, \dx
    \dy \nonumber \\
    &= \sum_{b \in \nodes} -\frac{1}{2} \sum_{\abs{\alpha}=k} c_\alpha
    \int_{\R^d \setminus U} \int_{\supp{(g_b)}} \h{w}(x-y) \abs{(D^\alpha g_b)(x)-(D^\alpha g_b)(y)}^2\, \dx
    \dy \nonumber \\
    &\leq \sum_{b \in \nodes} \abs{g_b}_{k,\R^d}^2 \label{double int number 1}.
\end{align}
Before calculating the final integral let us examine the following
expression for $b \in \nodes$ and $\alpha \in \Z^d_+$ with
$\abs{\alpha}=m$,
\begin{align}
    &\sum_{\substack{c \in \nodes \\ c \neq b}}  \int_{\supp{(g_c)}} \int_{\supp{(g_b)}} \h{w}(x-y) \abs{(D^\alpha
    g_b)(x)-(D^\alpha g_c)(y)}^2\, \dx \dy \nonumber \\
        &= \sum_{\substack{c \in \nodes \\ c \neq b}}  \int_{\supp{(g_c)}}
    \int_{\supp{(g_b)}} \h{w}(x-y) \abs{(D^\alpha
    g_b)(x)-(D^\alpha g_b)(y) + (D^\alpha
    g_c)(x)-(D^\alpha g_c)(y)}^2\, \dx \dy \nonumber \\
    &\leq
    2\sum_{\substack{c \in \nodes \\ c \neq b}}
    \int_{\supp{(g_c)}}
    \int_{\supp{(g_b)}} \h{w}(x-y) \abs{(D^\alpha
    g_b)(x)-(D^\alpha g_b)(y)}^2\, \dx \dy \nonumber \\
    & \hspace{3.5cm} +2\sum_{\substack{c \in \nodes \\ c \neq b}}
    \int_{\supp{(g_c)}}
    \int_{\supp{(g_b)}} \h{w}(x-y) \abs{(D^\alpha
    g_c)(x)-(D^\alpha g_c)(y)}^2\, \dx \dy \nonumber \\
    &\leq 2
    \int_{\R^d}
    \int_{\R^d} \h{w}(x-y) \abs{(D^\alpha
    g_b)(x)-(D^\alpha g_b)(y)}^2\, \dx \dy \nonumber \\
    & \hspace{3.5cm} +2\sum_{\substack{c \in \nodes \\ c \neq b}}
    \int_{\R^d}
    \int_{\supp{(g_b)}} \h{w}(x-y) \abs{(D^\alpha
    g_c)(x)-(D^\alpha g_c)(y)}^2\, \dx \dy \label{int supp c supp b}.
\end{align}
Finally, using the observation regarding the support of $g_a$, $a
\in \nodes$, once again and~\eqref{int supp c supp b} it follows
that
\begin{align}
-\frac{1}{2} &\sum_{\abs{\alpha}=k} c_\alpha \int_{U} \int_{U}
\h{w}(x-y) \biggabs{\sum_{a \in \nodes} ((D^\alpha g_a)(x)-(D^\alpha
    g_a)(y))}^2\, \dx\dy \nonumber \\
    &= \sum_{b \in \nodes} \sum_{c \in \nodes} -\frac{1}{2} \sum_{\abs{\alpha}=k} c_\alpha
    \int_{\supp{(g_c)}} \int_{\supp{(g_b)}} \h{w}(x-y) \biggabs{\sum_{a \in \nodes} ((D^\alpha
    g_a)(x)-(D^\alpha g_a)(y))}^2\, \dx \dy \nonumber\\
    &= \sum_{b \in \nodes} \sum_{c \in \nodes} -\frac{1}{2} \sum_{\abs{\alpha}=k} c_\alpha
    \int_{\supp{(g_c)}} \int_{\supp{(g_b)}} \h{w}(x-y) \abs{(D^\alpha
    g_b)(x)-(D^\alpha g_c)(y)}^2\, \dx \dy\nonumber\\
    &= \sum_{b \in \nodes} -\frac{1}{2} \sum_{\abs{\alpha}=k} c_\alpha
    \int_{\supp{(g_b)}} \int_{\supp{(g_b)}} \h{w}(x-y) \abs{(D^\alpha
    g_b)(x)-(D^\alpha g_b)(y)}^2\, \dx \dy\nonumber\\
    & \hspace{1.5cm} +\sum_{b \in \nodes} \sum_{\substack{c \in \nodes \\ c \neq b}} -\frac{1}{2} \sum_{\abs{\alpha}=k} c_\alpha
    \int_{\supp{(g_c)}} \int_{\supp{(g_b)}} \h{w}(x-y) \abs{(D^\alpha
    g_b)(x)-(D^\alpha g_c)(y)}^2\, \dx \dy\nonumber\\
    &\leq \sum_{b \in \nodes} \abs{g_b}_{k,\R^d}^2 + 2 \sum_{b \in \nodes}
    \abs{g_b}_{k,\R^d}^2+2 \sum_{c \in \nodes}
    \abs{g_c}_{k,\R^d}^2\nonumber\\
    &\leq 5 \sum_{b \in \nodes} \abs{g_b}_{k,\R^d}^2 \label{double int number 2}.
\end{align}
Substituting~\eqref{double int number 1} and~\eqref{double int
number 2} into~\eqref{split double integral} we find
\begin{equation*}
\biggabs{\sum_{a \in \nodes} g_a}_{k,\R^d}^2 \leq 7 \sum_{a \in
\nodes} \abs{g_a}_{k,\R^d}^2.
\end{equation*}
Hence, applying Condition 5 to the above inequality we have
\begin{equation*}
    \biggabs{\sum_{a \in \nodes} g_a}_{k,\R^d}^2 \leq 7 C_1^2 \sum_{a \in \nodes}
    \abs{f}_{k,B_a}^2\leq 7 C_1^2 \abs{f}_{k,\R^d}^2.
\end{equation*}
Now set $H= f-\sum_{a\in \nodes} g_a$. It then follows from
Condition 1 that $H(x) = (P_{(a,c')} f)(x)$ for all $x\in B_a$, and
from Condition 3 that $H(a) = f(a)$ for all $a\in \nodes$. Let $\phi
\in C^{\infty}_0(\R^{d})$ be supported on the unit ball and enjoy
the properties
\begin{equation*}
    \int_{\R^d} \phi(x)\,\dx = 1\qquad \mbox{and}\qquad \int_{\R^d}
    \phi(x) x^\alpha\,\dx =0,\qquad \mbox{for all}\ 0< \abs{\alpha} \leq
    k.
\end{equation*}
Now set $F=\phi_\delta * H$. Using Lemma~\ref{seminorm convolution
bound}, there is a constant $C_2>0$, independent of $q$ and $f$,
such that
\begin{align*}
        \abs{F}_{m,\R^d}^2
            \leq   C_2 \delta^{2(k-m)} \biggabs{f - \sum_{a \in \nodes} g_a}_{k,\R^d}^2
            &\leq   2 C_2 \delta^{2(k-m)} \biggl( \abs{f}_{k,\R^d}^2+ \biggabs{\sum_{a \in \nodes}
            g_a}_{k,\R^d}^2 \biggr)\\
            &\leq   2 C_2 (1+7 C_1^2) \delta^{2(k-m)} \abs{f}_{k,\R^d}^2.
\end{align*}
Similarly, there is a constant $C_3>0$, independent of $q$ and $f$,
such that
\begin{equation*}
        \abs{F}_{k,\R^d}^2
            \leq   C_3 \biggabs{f - \sum_{a \in \nodes} g_a}_{k,\R^d}^2
            \leq   2 C_3 (1+7 C_1^2)\abs{f}_{k,\R^d}^2.
\end{equation*} Thus
$\abs{F}_{m,\R^d} \leq C q^{k-m} \abs{f}_{k,\R^d}$ and
$\abs{F}_{k,\R^d} \leq  C \abs{f}_{k,\R^d}$ for some appropriate
constant $C>0$. Since $F=\phi_\delta
* H$ and $H\restrict{B_a}\in \Pi_k(\R^d)$ for each $a\in \nodes$, it
follows from Lemma~\ref{poly repro} that $F(a)=H(a) =f(a)$ for all
$a \in \nodes$.\qed

\begin{thm}
\label{main}
    Let $\Omega \subset \R^d$ be a V-domain and let $w:\R^d \To \R$ be a
    measurable function satisfying (W0){\Ndash}(W12). Let $k+\mu-d/2>0$
    and $m\geq k$. For each
    $h>0$, let $\nodes_h$ be a finite,
    $\Pi_m(\R^d)$-unisolvent subset of $\Omega$ with fill-distance $h$. Assume also that there is a
    quantity $\rho>0$ such that the mesh-ratio of each $\nodes_h$ is
    bounded by $\rho$ for all $h>0$. For each mapping
    $f:\nodes_h\rightarrow \R$, let $S_m^h f$ be the
    minimal norm interpolant to $f$ on $\nodes_h$ from $\CMcal{Z}^m(\R^d)$.
    Then there exists a constant $C>0$, independent of
    $h$, such that for all $f \in \CMcal{Y}^k(\Omega)$,
    \begin{equation*}
        \norm{f-S^h_m f}_{L_2(\Omega)} \leq C h^{k-\lambda/2-d/2}\abs{f}_{k,\Omega},\qquad \mbox{as $h \To 0$}.
    \end{equation*}
\end{thm}

\proof Take $f\in \CMcal{X}^k(\R^d)$. Construct $F$ in accordance
with Theorem~\ref{general result} and set $G=f-F$. Then $F(a)=f(a)$
and $G(a)=0$ for all $a \in \nodes_h$. Furthermore, there is a
constant $C_1>0$, independent of $f$ and $h$, such that
\begin{equation}
\abs{F}_{m,\R^d} \leq C_1 \Bigl(\frac{h}{\rho}\Bigr)^{k-m}
\abs{f}_{k,\R^d},\qquad \abs{G}_{k,\R^d} \leq
\abs{f}_{k,\R^d}+\abs{F}_{k,\R^d} \leq
(1+C_1)\abs{f}_{k,\R^d}.\label{F G bounds}
\end{equation}
Thus $S_m^h f = S_m^h F$ and $S_k^h G = 0$, where we have adopted
the obvious notation for $S_k^h$. Hence,
\begin{equation*}
    \norm{f-S_m^h f}_{L_2(\Omega)} =
    \norm{(F+G)-S_m^hF}_{L_2(\Omega)} \leq \norm{F-S_m^hF}_{L_2(\Omega)} +
    \norm{G-S_k^hG}_{L_2(\Omega)}.
\end{equation*}
Now, employing the error estimate in \cite{brownlee2}, there are
positive constants $C_2>0$ and $C_3>0$, independent of $h$ and $f$,
such that
\begin{equation*}
    \norm{f-S_m^h f}_{L_2(\Omega)} \leq C_2 h^{m-\lambda/2-d/2}
    \abs{F}_{m,\Omega}
    + C_3 h^{k-\lambda/2-d/2} \abs{G}_{k,\Omega},\qquad \mbox{as $h\rightarrow 0$.}
\end{equation*}
Finally, using the bounds in~\eqref{F G bounds} we have
\begin{equation}\label{nearly done}
    \norm{f-S_m^h f}_{L_2(\Omega)} \leq C_4 h^{k-\lambda/2-d/2}
    \abs{f}_{k,\R^d},\qquad \mbox{as $h\rightarrow 0$,}
\end{equation}
for some appropriate $C_4>0$. In particular,~\eqref{nearly done}
holds for all $f \in X^k(\R^d)$. As $\CMcal{Y}^k(\R^d)$ is a dense
linear subspace of $X^k(\R^d)$ then~\eqref{nearly done} extends to
hold for all $f \in \CMcal{Y}^k(\R^d)$ using a standard normed space
argument \cite[Page 180]{jameson}. To complete the proof we now let
$f \in \CMcal{Y}^k(\Omega)$ and define $f^\Omega$ in accordance with
Theorem~\ref{vail norm exten thm 2}. It follows that there is a
$C_5>0$ such that
\begin{equation*}
\norm{f-S_m^h f}_{L_2(\Omega)} \leq C_4 h^{k-\lambda/2-d/2}
    \abs{f^\Omega}_{k,\R^d} \leq C_4 C_5 h^{k-\lambda/2-d/2}
    \abs{f}_{k,\Omega},\qquad \mbox{as $h\rightarrow 0$.} \qedhere
\end{equation*}

\section*{Acknowledgements} It is a pleasure to acknowledge
that this paper is a generalisation of joint work with Will Light.
The author would like to dedicate this paper to Will's memory.

\end{document}